\documentclass[11pt]{amsart}
\usepackage[linktocpage]{hyperref}
\usepackage{amssymb, paralist, xspace, graphicx, url, amscd, euscript, mathrsfs,stmaryrd,epic,eepic,color}
\usepackage[all]{xy}
\usepackage{amsthm}
\usepackage{enumerate}
\usepackage{multirow}
\usepackage{colortbl}
\usepackage{comment}
\definecolor{kugray5}{RGB}{224,224,224}
\usepackage{lipsum}
\usepackage{multirow}
\usepackage{tikz-cd}
\usetikzlibrary{arrows,shapes,chains}
\SelectTips{cm}{}

\numberwithin{equation}{section}
1
\numberwithin{subsection}{section}

\allowdisplaybreaks[1]

\newtheorem*{namedtheorem}{\theoremname}
\newcommand{\theoremname}{testing}

\newtheorem{theorem}[subsection]{Theorem}
\newtheorem{proposition}[subsection]{Proposition}
\newtheorem{proposition-definition}[subsection]
{Proposition-Definition}
\newtheorem{corollary}[subsection]{Corollary}
\newtheorem{lemma}[subsection]{Lemma}
\theoremstyle{definition}

\newtheorem{remark}[subsection]{Remark}

\theoremstyle{remark}

\newcommand\cB{\mathcal{B}}

\newcommand\cF{\mathcal{F}}

\newcommand\cM{\mathcal{M}}

\newcommand\cO{\mathcal{O}}
\newcommand\cP{\mathcal{P}}

\newcommand\cU{\mathcal{U}}

\newcommand\CC{\mathbb{C}}
\newcommand\DD{\mathbb{D}}

\newcommand\LL{\mathbb{L}}

\newcommand\PP{\mathbb{P}}
\newcommand\QQ{\mathbb{Q}}
\newcommand\RR{\mathbb{R}}

\newcommand\ZZ{\mathbb{Z}}

\newcommand\fE{\mathfrak{E}}

\newcommand\frg{\mathfrak{g}}

\newcommand\re{{\rm Re}}

\newcommand\Pic{{\rm Pic}}

\newcommand\Proj{{\rm Proj}}

\newcommand{\uu}{$\ddot{\hbox{u}}$}
\newcommand{\q}{/\!\!/}

\theoremstyle{plain}
\theoremstyle{definition}
\newtheorem{defn}{Definition}

\begin{document}

\title{Moduli space of quasi-polarized K3 surfaces of degree 6 and 8}
\author{Zhiyuan Li,~~Zhiyu Tian}

\address{Fudan University, Shanghai Center For Mathematical Science,
Shanghai, 200433, China}
\email{zhiyuan\_li@fudan.edu.cn}

\address{Peking University, Beijing International Center for Mathematical Research, Beijing, 100871, China}

\email{zhiyutian@bicmr.pku.edu.cn}

\begin{abstract}
In this paper, we study the moduli space of quasi-polarized complex K3 surfaces of degree $6$ and $8$ via geometric invariant theory. The general members in such moduli spaces are complete intersections in projective spaces and we have natural GIT constructions for the corresponding moduli spaces and we show that the K3 surfaces with at worst ADE singularities are GIT stable. We give a concrete description of boundary of the compactification of the degree $6$ case via the Hilbert-Mumford criterion. We compute the Picard group via Noether-Lefschetz theory and discuss the connection to the Looijenga's compactifications from arithmetic perspective. One of the main ingredients is the study of the projective models of K3 surfaces in terms of Noether-Lefschetz divisors.
\end{abstract}

\maketitle

\section{Introduction} A primitively quasi-polarized K3 surface $(S,L)$ of degree $2\ell$  over $\CC$ consists of a complex K3 surfaces, a big and nef line bundle $L$ such that $c_1(L)\in H^2(S,\ZZ)$ is a primitive class and $L^2=2\ell$. Let $\cF_{2\ell}$ be the moduli space of primitively quasi-polarized complex  $K3$ surfaces of degree $2\ell$. It is well-known that the period map behaves very well on $\cF_{2\ell}$. Namely, if we denote by $\DD$ the  period domain of K3 surfaces  and $\Gamma_{2\ell}$ the monodromy group, global Torelli theorem tells us  that  $\cF_{2\ell}$ is isomorphic to $\Gamma_{2\ell}\backslash \DD$ via the period map.

Besides the Hodge theoretical construction,  there are also explicit algebraic construction of  $\cF_{2\ell}$  via geometric invariant theory (GIT) for low degree K3 surfaces, where such a general K3 surface is a complete intersection in the projective space. For instance, the GIT construction of $\cF_2$ and $\cF_4$ has been worked out by Shah (cf.~\cite{Sh80}\cite{Sh81}). 
When $2\ell=6$ or $8$, a general element $(S,L)\in \cF_{2\ell}$ is a complete intersection of a smooth quadric and a cubic in $\PP^4$ or a complete intersection of three quadrics in $\PP^5$ respectively.  

In this paper, we describe the GIT construction of moduli space of these complete intersects and characterize the image of period map for  such complete intersections as a complement of certain Noether-Lefschetz (NL) divisors in $\Gamma_{2\ell}\backslash\DD$. The latter one has a natural arithmetic compactification constructed by  Looijenga (cf.~\cite{Lo03}), and we will compare this arithmetic compactification with the natural GIT compactification.

More precisely, for any non-negative integers $d,g$, the NL-divisor $D_{d,g}^{2\ell}\subset \cF_{2\ell}\cong \Gamma_{2\ell}\backslash \DD$ is defined to be the locus of  quasi-polarized complex K3 surfaces $(S,L)\in \cF_{2\ell}$ such that $\Pic(S)$ contains a rank two primitive sublattice of the following form:
\begin{equation}\label{eq1.1} \left.\begin{array}{c|c|c}~ & L & \beta \\\hline L & 2\ell & d \\\hline \beta & d & 2g-2\end{array}\right. \end{equation}
for some $\beta\in \Pic(S)$.  For simplicity of notations, we identify $D^{2\ell}_{d,g}$ as divisors on $\Gamma_{2\ell}\backslash \DD$ via period map. One of our main result is:
\begin{theorem}\label{thm1}For $\ell=3,4$, 
the complete intersections in $\PP^{\ell+1}$ of degree $2\ell$ with at worst simple singularities (i.e. isolated ADE singularities) are GIT-stable. Let $\cM_{2\ell}$ be the moduli space of such complete intersections with at worst simple singularities.  Then the period map extends to $\cM_{2\ell}$ and  its image in $\Gamma_{2\ell}\backslash \DD$ is the complement of $D^{2\ell}_{1,1}, D^{2\ell}_{2,1}$ and $D^{2\ell}_{3,1}$.

Furthermore, the natural GIT compactification  $\overline{\cM}_{2\ell}$ is not isomorphic to Looijenga's compactification of the complement $ \Gamma_{2\ell}\backslash\DD-\bigcup\limits_{d=1}^3 D^{2\ell}_{d,1}$.
\end{theorem}
\begin{remark}We refer the readers to 
\cite{BL15} for the analysis of GIT stability  for such complete intersections with semi log canonical singularities. 
\end{remark}

Secondly, we have classified the boundary components of $\cM_6$ in its GIT compactification $\overline{\cM}_{6}$. The main result is  as follows:
\begin{theorem}\label{thmboundary}
The boundary $\overline{\cM}_6\backslash \cM_6$ consists of $9$ irreducible components whose general member $X$ is described as follows:
 \begin{enumerate}
     \item [$\alpha)$] $(\mathbf{Semitable})$    $X$ has two corank $3$ singularities, but not a union of a quadric surface and a quadric cone with multiplicity two.  
      \item  [$\widetilde{\beta})$] $(\mathbf{Stable})$   
        $X$ is a union of a smooth quadric surface and a smooth complete intersection of two quadrics.
      \item    [$\gamma)$]   $(\mathbf{Semistable})$  $X$ has two simple elliptic singularities of type $\tilde{E}_8$, whose projective tangent cone intersect $X$ along lines, and not  the union of three quadric cones. 
      
      \item  [$\delta)$] $(\mathbf{Stable})$  $X$  has an isolated $\widetilde{E}_7$ singularity. 
       \item [$\epsilon)$]  $(\mathbf{Stable})$   $X$ has an isolated $\widetilde{E}_8$ singularity, whose projective tangent cone meets $X$ at a point.
      
      \item   [$\zeta)$] $(\mathbf{Stable})$   $X$ is  singular along a line. 
      \item  [$\eta)$]  $(\mathbf{Stable})$   $X$ is singular along a conic. 
      \item  [$\theta)$]  $(\mathbf{Stable})$   $X$ is singular along a twisted cubic. 
      \item  [$\phi)$] $(\mathbf{Stable})$  $X$ is singular along a rational normal curve of degree $4$.
 \end{enumerate}
The stratum $\alpha$ is $6$-dimensional,  $\widetilde{\beta}$ is $7$-dimensional (it contains a $2$-dimensional semistable loci $\beta$),    $\gamma$ and $\phi$  are $2$-dimensional,   $\delta$ and $\zeta$ are $11$-dimensional, $\epsilon$ is $8$-dimensional, $\eta$ is $7$-dimensional, $\theta$ is $3$-dimensional.
\end{theorem}
According to the work of \cite{St95}, the Baily-Borel compactification  $(\Gamma_6\backslash \DD)^{\ast}$ of the Shimura variety $\Gamma_6\backslash \DD$ consists of $10$ irreducible components. The extended period map induces a birational map 
\begin{equation}\label{ep}
\overline{\cM_6}\dashrightarrow (\Gamma_6\backslash \DD)^{\ast}.    
\end{equation}
 According to the spirit of Hassett-Keel-Looijenga program raised by \cite{La09}, it is expected that the map \eqref{ep} can factor through a sequence of elementary birational transformations of Shimura type, i.e. the exceptional loci comes from Shimura subvarieties. This problem will be solved in a forthcoming paper \cite{GLRST}.  

In \cite{MP12}, Maulik and Pandharipande  have conjectured that the Picard group of $\cF_{2\ell}$ with $\QQ$-coefficients is spanned by the NL-divisors $\{D_{d,g}^{2\ell}\}$ on $\cF_{2\ell}$. This conjecture has been verified in \cite{BLMM17} via automorphic representation theory and finding a geometric approach remains highly interesting.  Here,  using the main theorem, we can  compute the Picard group of $\cF_{6}$ and $\cF_8$ from the GIT construction. 
\begin{corollary}\label{nlthm}
When $2\ell=6$ or $8$, the Picard group $\Pic_\QQ(\cF_{2\ell})$ with rational coefficients is spanned by NL-divisors $D^{2\ell}_{d, 1}, d=1,2,3,4$. Moreover,  
$$\dim_\QQ \Pic_\QQ(\cF_{2\ell})=\dim H^2(\Gamma_{2\ell},\QQ)=4$$ 
for $2\ell=6$ or $8$.
\end{corollary}
The first part of this result has been also obtained by K. O'grady in \cite{O86} using a slightly different method. We just point out that the similar  approach  has been applied to  K3 surfaces with Mukai models (i.e. $10 \leq 2\ell\leq 18$, or $2\ell=22$) in \cite{GLT15}.

\subsection*{Acknowledgement}We are grateful to O'Grady and Laza for many useful comments. The first author is  supported by NSFC grants for  General Program (11771086), Key Program (11731004) and the Shu Guang Program (17SG01) of Shanghai Education Commission. The second author is partially supported by the program``Recruitment of global experts", and NSFC grants No. 11871155, No. 11831013, No.11890662.
 
\section{Noether-Lefschetz divisors for K3 surfaces} 
Let us recall the Noether-Lefschetz theory on K3 surfaces. 
\subsection{Noether-Lefschetz divisors} Let $(S,L)$ be a primitively quasi polarized K3 surface of degree $2\ell$. The middle cohomology $\Lambda:=H^2(S,\ZZ)$ is a unimodular even lattice of signature $(3,19)$ under the intersection form $\left<,\right>$.  Let $h_{2\ell}=c_1(L)$, then the orthogonal complement 
$\Lambda_{2\ell}:=h_{2\ell}^{\perp}\subset \Lambda$ is an even lattice of signature  $(2,19)$, which has a unique representation:
\begin{equation}
\Lambda_{2\ell}=\ZZ\omega \oplus U^{\oplus 2}\oplus E_8(-1)^{\oplus 2},
\end{equation}
where $\left<\omega,\omega\right>=-2\ell$, $U$ is the hyperbolic plane and $E_8(-1)$ is the unimodular, negative definite even lattice of rank $8$.

Let $\Lambda_{2\ell}^\CC=\Lambda_{2\ell}\otimes \CC$. The period domain $\DD$ associated to $\Lambda_{2\ell}$ can be  realized as a connected component of 
$$\DD^{\pm}:=\{v\in \PP(\Lambda_{2\ell}^\CC)| \left<v,v\right>=0,-\left<v, \bar{v}\right> > 0\}.$$
The monodromy group 
$$\Gamma_{2\ell}=\{g\in \text{Aut}(\Lambda_{2\ell})^{+}|~g~\hbox{acts trivially on}~\Lambda^\vee_{2\ell}/\Lambda_{2\ell} \},$$
naturally acts on $\DD,$ where $\text{Aut}(\Lambda_{2\ell})^{+}$ is the identity component of $\text{Aut}(\Lambda_{2\ell})$.
According to the Global Torelli theorem of K3 surfaces, there is an isomorphism $$\cF_{2\ell}\cong \Gamma_{2\ell}\backslash\DD$$ via the period map. Then $\cF_{2\ell}$ is a locally Hermitian symmetric variety with only quotient singularities, and hence $\QQ$-factorial.

The NL-divisor $D_{d,g}^{2\ell}$ can be identified as the quotient of the  union of  subdomains on $\DD$ as following: 
\begin{equation}\label{eq2.1} D_{d,g}^{2\ell}\cong \Gamma_{2\ell}\backslash \bigcup\limits_{M\subset \Lambda} \{z\in \DD~| \left<z,m\right>=0,\forall m\in M\}, \end{equation}
where $M$ runs for all rank two primitive sublattice of $\Lambda_{2\ell}$ of  the form \eqref{eq1.1}. In the language of Heegner divisors, the right hand side of \eqref{eq2.1}  is called  the arithmetic quotient of {\it hyperplane arrangement} in $\DD$. As known in \cite{O86} Proposition 1.3, we have the irreducibility theorem:
\begin{theorem}
All the NL-divisors $D_{d,g}^{2\ell}\in\Pic_\QQ(\cF_{2\ell})$ are irreducible.
\end{theorem}

\begin{remark}
The definition of NL-divisors we used here is slightly different from the one in \cite{MP12}. Maulik and Pandharipande define the NL-divisors  without the assumption of primitivity of the sublattice $M$ in (\ref{eq2.1}). But the span of these divisors are the same as ours (cf.~\cite[\S 0.2]{MP12} ). 
\end{remark}

\subsection{Dimension formula}
Let us denote by $\Pic_\QQ(\Gamma_{2\ell}\backslash\DD)^{NL}$ the subgroup of $\Pic_\QQ(\Gamma_{2\ell}\backslash\DD)$ generated by NL-divisors with $\QQ$-coefficients. By \cite{Br02,MP12}, we know that the dimension $\rho_{2\ell}$ of the span of Heegner divisors on $\Gamma_{2\ell}\backslash \DD$ can be explicitly computed by the following formula:
 \begin{equation}
 \label{eq:1.2}
 \begin{aligned}
 \rho_{2\ell}&=\frac{31 }{24}l+\frac{55}{24}-\frac{1}{6\sqrt{6l}} \re (e^{\frac{5\pi i}{12}}(G(-1,4l)+G(3,4l)))\\&-\frac{1}{4\sqrt{2\ell}}\re (G(-1,2\ell))-\sum\limits_{k=0}^{l}\{\frac{k^2}{4l}\}-\sharp\{k~|~\frac{k^2}{4l}\in\ZZ, 0\leq k\leq l \} 
 \end{aligned}
 \end{equation}
where $\{,\}$ denotes the fraction part  and  $G(a,b)$ is the generalized quadratic Gauss sum: $$G(a,b)=\sum\limits_{k=0}^{b-1} e^{2\pi i \frac{ak^2}{b}}.$$
Denote by $d_{Eis}=\sharp\{k~|~\frac{k^2}{4l}\in\ZZ, 0\leq k\leq l \} $. After applying the summation formula proved by Gauss in 1811 (cf.~\cite[\S 2.2]{BE81} ), one can simply get
\begin{lemma}
\begin{equation}\label{eq1.4}\rho_{2\ell}= \frac{31l+55}{24}-\frac{1}{4}\alpha_l-\frac{1}{6}\beta_l -  \sum\limits_{k=0}^{l}\{\frac{k^2}{4l}\}-d_{Eis},
\end{equation}
where 
$$\alpha_l= \begin{cases}
      \left( \frac{2\ell}{2\ell-1}\right), & \text{ $l$ is even }; \\
     0& \text{otherwise}.
\end{cases}, ~~~~\beta_l= \begin{cases}
     \left( \frac{l}{4l-1}\right)-1, & \text{ if $3| l$}; \\
   \left( \frac{l}{4l-1}\right)+\left( \frac{l}{3}\right)   & \text{otherwise}.
\end{cases}$$ and $\left( \frac{a}{b}\right)$ is the Jacobi symbol.
\end{lemma}
As shown in \cite{MP12}, the span of NL-divisors are the same as the span of non irreducible  divisors on $\Gamma_{2\ell}\backslash\DD$. 

\subsection{Projective models of K3 surfaces}
Let $(S,L)$ be a smooth K3 surface with a primitive quasi-polarization $L$ of degree $2\ell$. The linear system $|L|$ defines a map  $\psi_L$ from $S$ to $\PP^{l+1}$. The image of $\psi_L$ is called a {\it projective model} of $S$.

In \cite{Sa74}, Saint-Donat gives a precise description of all projective models of $(S,L)$ when $\psi_L$ is not a birational morphism. 
\begin{proposition}\label{prop:2.3}\cite{Sa74} Let $L$ be the primitive quasi-polarization of degree $2\ell$ on $S$ and let $\psi_L$ be the map defined by $|L|$. Then there are  following possibilities:
\begin{enumerate}
  \item $\psi_L$ is birational to a degree $2\ell$ surface in $\PP^{l+1}$. In particular, $\psi_L$ is a closed embedding when $L$ is ample. 
  \item  $\psi_L$ is a generically $2:1$ map and $\psi_L(S)$ is a smooth rational normal scroll of degree $l$,  or a cone over a rational normal curve of degree $l$.
  \item $|L|$ has  a fixed component $D$, which is a smooth rational curve. Moreover, $\psi_L(S)$ is a rational normal curve of degree $l+1$ in $\PP^{l+1}$.
\end{enumerate}
\end{proposition}
We call K3 surfaces of type $(1)$, $(2)$, $(3)$  {\it nonhyperelliptic}, {\it unigonal}, and {\it digonal} K3 surfaces accordingly.
When $l=2,3,4$, the projective model of a general quasi-polarized K3 surface $(S, L)$ is a complete intersection in the projective space $\PP^{l+1}$.
\begin{remark}
Assume that $\psi_L$ is a birational morphism. Then one can easily see that $L$ is not ample if and only if there exists an exceptional $(-2)$ curve $D\subseteq S$.  The morphism $\psi_L$ will  factor through a contraction $\pi:S\rightarrow \tilde{S}$ where $\tilde{S}$ is a singular K3 surface with simple singularities. 

 Recalling that the NL-divisor $D^{2\ell}_{0,0}$ parametrizes all K3 surfaces  $(S,L)$ of degree $2\ell$ with exceptional $(-2)$ curves. Therefore,  the projective model of a general member in $D^{2\ell}_{0,0}$ is a surface in $\PP^{l+1}$ of degree $2\ell$ with simple singularities.  
\end{remark}

In this paper, we mainly consider the case $2\ell=6$ and $8$, where the classification of projective models of $S$ can be read off from the Picard lattice of $S$.
\begin{lemma}\label{lem:1}
Let $(S,L)$ be a smooth quasi-polarized K3 surface of degree $2\ell$ ($2\ell=6$ or $8$). Then
\begin{enumerate}
 \item  $(S,L)\in D_{1,1}^{2\ell}$ if and only if $S$ is digonal except
 \begin{enumerate}[$(\ast)$]
 \item $L^2=8$ and $L=L'+E+C$, where $C$ is a rational curve, $E$ is an irreducible elliptic curve  and $L'$ is irreducible of genus two with $L'\cdot C=E\cdot C=1 $ and $L'\cdot E=2$. The image $\psi_L(S)$ is contained in a cone over cubic surface in $\PP^4$.
 \end{enumerate}
\item  $(S,L)\in D_{2,1}^{2\ell}$ if and only if  $S$ is unigonal.
\item $(S,L)\in D_{3,1}^{2\ell}$ if and only if $S$ is one of the following:

\begin{itemize}
\item when $l=3$, $S$ is birational  to the complete intersection of a singular quadric and a cubic in $\PP^4$ via $\psi_L$.
\item when $l=4$, $S$ is  either birational to a bidegree $(2,3)$ hypersurface of the Serge variety 
$\PP^1\times\PP^2\hookrightarrow \PP^5$ via $\psi_L$ or is in case $(\ast)$.
\end{itemize}
\end{enumerate}
\end{lemma}

\begin{proof} 
The proof of (1) and (2) are straightforward from Proposition \ref{prop:2.3}. See also \cite[\S 2, \S 5 ]{Sa74} for more detailed discussion. 

Now we suppose that a quasi-polarized K3 surface $(S,L)\in D_{3,1}^6 $ is  neither unigonal nor diagonal. Then $\psi_L$ is a birational map to a complete intersection of a quadric and a cubic.  Our first statement of (3) comes from the fact any quadric threefold containing a plane cubic must be singular. If $(S,L)\in D_{3,1}^8$,  the assertion follows from \cite{Sa74} Proposition 7.15 and  Example 7.19. 
\end{proof}

\begin{remark}
We would like to refer the readers to \cite{JK04} and \cite{GLT15} for a detailed description of projective models of low degree ($2\ell\leq 22$) K3 surfaces. 
\end{remark}
\begin{section}{Complete intersection of a quadric and a cubic}
In this section, we construct the moduli space of  the complete intersection of a smooth quadric and a cubic in $\PP^4$ via geometric invariant theory. 

\subsection{Terminology and Notations} In the rest of this paper, we will use the following terminology. 
Let $f(u,v,w)$ be an analytic function in $\CC[[u,v,w]]$ whose leading term defines an isolated singularity at the origin. We have the following types of singularities:  
\noindent 
\begin{itemize}
  \item Simple singualrities: isolated $A_n$, $D_k$, $E_r$ singularities.
  \item Simple elliptic singularities $\tilde{E}_r$:
  \begin{itemize}
  \item $\tilde{E}_6$:  $f=u^3+v^3+w^3+auvw$, 
  \item $ \tilde{E}_7$: $f=u^2+v^4+w^4+auvw$,
   \item $ \tilde{E}_8$: $f=u^2+v^3+w^6+auvw$, 
  \end{itemize}
\end{itemize}

We will use the notation $l(x),q(x),c(x)$ as linear, quadratic and cubic polynomials of $x=(x_0,\ldots,x_n)$.
\subsection{Cubic sections on quadric threefolds}Let $Q$ be the smooth quadric threefold in $\PP^4$ defined by the equation $$x_0x_4+x_1x_3+x_2^2=0.$$
Since every nonsingular quadric hypersurface in $\PP^4$ is projectively equivalent to $Q$,  a complete intersection of a smooth quadric and a cubic can be identified with an element in $|\cO_Q(3)|$. 

The automorphism group of $Q$ is the reductive Lie group $SO(Q)(\CC)$ which is isomorphic to $SO(5)(\CC)$.  Then we can naturally describe the moduli space of the complete intersection of a smooth quadric and a cubic as the GIT quotient of the linear system $|\cO_{Q}(3)|=\PP(V)$, where $V$ is a $30$-dimensional vector space defined by the exact sequence
$$0\rightarrow H^0(\PP^4,\cO_{\PP^4}(1))\rightarrow H^0(\PP^4,\cO_{\PP^4}(3))\rightarrow V\rightarrow 0.$$
Let us take the set of monomials \begin{equation}\label{eq:3.1}\cB:=\{x_0^{a_0}x_1^{a_1}\ldots x_4^{a_4}|\sum\limits_{i=0}^4 a_i=3 \text{ and }  a_0a_4=0 \}.
\end{equation}
to be a basis of $V$.
Sometimes, we may change the basis for simpler computations.
\subsection{Numerical criterion} Now we classify stability of  the points in $\PP(V)$ under the action of $SO(Q)(\CC)$ by applying the Hilbert-Mumford numerical criterion \cite{MFK94}.

As is customary, a one parameter subgroup (1-PS) of $SO(Q)(\CC)$ can be diagonalized as $$ \lambda_{u,v}:t\in\CC^\ast\rightarrow \hbox{diag}(t^u,t^v,1,t^{-v},t^{-u}),$$  for some $u,v\in\ZZ$. We call such $\lambda_{u,v}:\CC^\ast \rightarrow SO(Q)(\CC)$ a {\it normalized} 1-PS of $SO(Q)(\CC)$ if $u\geq v\geq 0$.

Let $\lambda_{u,v}$ be a normalized 1-PS of $SO(Q)(\CC)$. Then the weight of a monomial $x_0^{a_0}x_1^{a_1}\ldots x_4^{a_4}\in \cB$ with respect to $\lambda_{u,v}$ is \begin{equation}(a_0-a_4)u+(a_1-a_3)v.\end{equation}
If we denote by $M_{\leq0}(\lambda_{u,v})$ (resp.~$M_{<0}(\lambda_{u,v})$) the set of monomials of degree $3$ which have non-positive (resp.~negative) weight with respect to $\lambda_{u,v}$, one can easily compute the maximal subsets $M_{\leq 0}(\lambda_{u,v})$ (resp. $M_{<0}(\lambda_{u,v})$ ), as listed in Table 1 (resp. Table 2) .

\begin{table}[ht] \label{tab:1}
\centering
\caption{Maximal subsets $M_{\leq 0}(\lambda)$}
  \begin{tabular}{|cl|c|c|c|}
    \hline
   & Cases &$(u,v)$& Maximal monomials  \\ \hline
    & (N1)&(1,0) & $x_1^{a_1}x_2^{a_2}x_3^{a_3},  \sum a_i=3$ \\ \hline
    &(N2)& (1,1) &$ x_0x_2x_3, x_2^3$ \\ \hline
     &(N3)&(2,1)& $ x_0x_3^2,x_1^2x_4, x_1x_2x_3, x_2^3$\\ \hline
  \end{tabular}
\end{table}

\begin{table}[ht]\label{tab:2}
\caption{Maximal subsets $M_{< 0}(\lambda)$}
\centering
  \begin{tabular}{ |c| c|c |}
    \hline
     Cases &$(u,v)$& Maximal monomials  \\ \hline
    (U1)& (1,1) &$ x_0x_3^2, x_2^2x_3$ \\ \hline
     (U2)& (3,1) &$ x_1^2x_4, x_1x_3^2, x_2^2x_3 $ \\ \hline
  \end{tabular}
\end{table}

According to the Hilbert-Mumford criterion,  an element $f(x_0,\ldots,x_4)\in \PP(V)$ is not properly stable (resp. unstable) if and only if  the weight of all monomial in $f$ is non-positive (resp. negative) for some $1$-PS.
Thus we obtain:
\begin{lemma}\label{lem:2}Let $X$ be the surface defined by an element in $\PP(V)$. Then $X$  is not properly stable  if and only if $X=Q\cap Y$ for some cubic  hypersurface $Y\subseteq \PP^4$ defined by a cubic polynomial in one of following cases:
\begin{itemize}
  \item $c(x_1,x_2,x_3,x_4)$;
  \item $x_0x_3l(x_2,x_3)+x_1x_2\ell_1(x_3,x_4)+x_1q(x_3,x_4)+c(x_2,x_3,x_4)$;
  \item $x_0x_3^2+x_1x_3l_1(x_2,x_3)+x_1x_4l_2(x_1,x_2,x_3,x_4)+c(x_2,x_3,x_4)$.
\end{itemize} 
\end{lemma}

For $f\in \PP(V)$ not properly stable,  using the destabilizing 1-PS $\lambda$,  the limit $\lim\limits_{t\rightarrow 0} f_t=f_0$ exists and it is invariant with respect to $\lambda$. The invariant part of polynomials of type $(N1)-(N3)$ are the followings:
\begin{enumerate}
 \item[($\alpha$)] $c(x_1,x_2,x_3)=0 $;
\item[($\beta$)] $\lambda_1 x_2^3+\lambda_2 x_1x_2x_3+\lambda_3 x_0x_2x_3+\lambda_4 x_1x_2x_4=0$,  $\lambda_i\in\CC$;
 \item[($\gamma$)] $\lambda_1x_2^3+ \lambda_2x_1x_2x_3 +\lambda_3x_0x_3^2+\lambda_4x_1^2x_4=0$, $\lambda_i\in\CC$.
 \\
\end{enumerate}

Similarly, we get
\begin{lemma}\label{lem:3} With the notation above, 
$X$ is not semistable if and only if $X=Q\cap Y$ for some cubic hypersurface $Y$ defined by one of the following equations:
\begin{itemize}
  \item $x_0x_3^2+x_1q(x_3,x_4)+c(x_2,x_3,x_4)$, and $c(x_2,x_3,x_4)$ has no $x_2^3$ term;
   \item $x_4q_1(x_1,x_2,x_3,x_4)+x_3q_2(x_2,x_3)+\lambda x_1x_3^2$.
\end{itemize} 
\end{lemma}


\subsection{Geometric interpretation of stability}  We use the terminology of the corank of the hypersurface singularities as in \cite{AGV12} and \cite{La09}.
\begin{defn}
Let $0\in \CC^n$ be a hypersurface singularity given by an equation $f(z_1,\ldots,z_n)=0$. The corank of $0$ is $n$ minus the rank of the Hessian of $f(z_1,\ldots,z_n)$ at $0$. 
\end{defn}

\begin{theorem}\label{thm:stable}
A complete intersection $X=Q\cap Y$ is not properly stable if and only if $X$ satisfies one of the following conditions:
\begin{enumerate}[(i)]
  \item $X$ has a  hypersurface singularity of corank $3$.
  \item $X$ is singular along a line $L$ and  there exists a plane $P$ such that  $P \cap Q=2L$  and $P$ is  contained in the projective tangent cone $\PP(CT_p(X))$ for any point $p\in L$.
  \item $X$ has a singularity $p$ which deforms to a singularity of $\widetilde{E}_8$ class, and the restriction of  the projective cone $\PP(CT_{p}(X))$  to $X$ contains a line $L$ passing through $p$ with multiplicity at least $6$.  
\end{enumerate}
\end{theorem}

\begin{proof}  As a consequence of Lemma \ref{lem:2}, it suffices to find the geometric characterizations of the complete intersections of type $(N1)-(N3)$. Here we do it case by case.

(i). If  $X$ is of type $(N1)$, then $X$ can be considered as the intersection of $Q$ and a cubic cone $Y$ with the vertex $p_0=[1,0,0,0,0]\in Q$. It is easy to see that  $p_0$ is a corank of  3 singularity of $X$.

 Conversely,  we write the equation of $Y$ as
$$x_0q(x_0,x_1,x_2,x_3)+c(x_1,x_2,x_3,x_4)=0.$$   
If we choose the affine coordinate
 \begin{equation}\label{eq:3.0}
  y_i:=x_i/x_0,
\end{equation} then the affine equation near $p_0$ is \begin{equation}\label{eq:3.2}
q(1,y_1,y_2,y_3)+c(y_1,y_2,y_3,-y_2^2-y_1y_3)=0.
\end{equation} in $\CC^3$.
It  has a corank $3$ singularity at the origin if and only if the quadric $q$ is $0$.
\\

(ii). If $X$ is of type $(N2)$,  then the equation of $Y$ is given by
$$ x_0x_3l(x_2,x_3)+x_1x_2\ell_1(x_3,x_4)+x_1q(x_3,x_4)+c(x_2,x_3,x_4),$$
and therefore $X$ is singular along the line  $L: x_2=x_3=x_4=0$. 

Moreover, for any point $p=[z_0,z_1,0,0,0]\in L$, the projective tangent cone $\PP(CT_p(X))$ at $p$ is defined as
\begin{equation} 
z_0x_4+z_1x_3= z_0x_3l(x_2,x_3)+z_1(x_2\ell_1(x_3,x_4)+q(x_2,x_3))=0,
\end{equation}
which contains the plane $P:x_3=x_4=0$ for each $p\in L$ and $P\cap Q=2L$.

Conversely, since the intersection of  $P$ and $Q$ is a double line $L$, we may certainly assume that the plane $P$ is defined by $$x_3=x_4=0$$ after some coordinate transform persevering the quadric form $Q$. Then the line $L=P\cap Q$ is given by $x_2=x_3=x_4=0$.

Because $X$ is singular along $L$, the equation of $Y$ can be written as:
 \begin{equation}
 x_0q_1(x_2,x_3)+x_1q_2(x_2,x_3,x_4)+c(x_2,x_3,x_4)=0.
 \end{equation}
  Then the projective tangent cone 
 $$\PP(CT_p (X))=\{z_0x_4+z_1x_3=z_0q_1(x_2,x_3,x_4)+z_1q_2(x_2,x_3,x_4)=0\}$$ 
 contains the plane $P$ for each point $p=[z_0,z_1,0,0,0]\in L$ only if the quadrics $q_i$ have no $x_2^2$ term. 
\\

(iii). For $X$ of type $(N_3)$, a similar discussion is as follows:  if $Y$ is defined by 
\begin{equation} 
x_0x_3^2+x_1x_3l_1(x_2,x_3)+x_1x_4l_2(x_1,x_2,x_3)+c(x_2,x_3,x_4)=0,
\end{equation}
then $X=Q\cap Y$ is singular at $p_0$.  After choosing the affine coordinates as (\ref{eq:3.0}), the affine equation near $p_0$ is  
\begin{equation}\label{eq:3.6} 
y_3^2+y_1y_3^2 f(y_1,y_2,y_3)+y_1y_2^2\ell(y_1,y_2)+ay_1y_2y_3+g(y_2,y_3)=0
\end{equation} 
for some polynomials $\ell, f,g$ with $\ell$ linear, $deg(f)\geq 1, deg(g)\geq 3$. Therefore, $p_0$ is a hypersurface singularity of corank $2$ and its projective tangent cone is a double plane $2P:x_3^2=x_4=0$. The remaining part is  straightforward. 

Conversely,  we take $p_0$ to be the isolated singular point which deforms to a singularity of $\widetilde{E}_8$ class. As it has corank at least $2$,  the equation of $Y$  can be written as  
$$ x_0q_1(x_1,\ldots,x_3)+x_1q_2(x_1,\ldots,x_4)+c(x_2,x_3,x_4)=0.$$ 
Then the quadric $ q_1(x_1,x_2,x_3)$ is of the form $l(x_1,x_2,x_3)^2$ for some linear polynomial $l$ because  $p_0$ is singular of  corank at least $2$.  

After we make a coordinate change preserving $Q$ and $p_0$,  the defining equation of $Y$ has two possibilities: 
\begin{enumerate}
\item  $x_0x_2^2+x_1q(x_1,x_2,x_3,x_4)+c(x_2,x_3,x_4)=0,$
\item $x_0x_3^2+x_1q(x_1,x_2,x_3,x_4)+c(x_2,x_3,x_4)=0.$
\end{enumerate}
The projective tangent cone at $\PP(CT_{p_0}(X))$ is a double plane 
$$2P:x_4=x_2^2=0,\hbox{ or }x_4=x_3^2=0.$$ 
The line $L$ contained in the  restriction of $2P$ to $X$ has to be defined by $x_2=x_3=x_4=0$. It follows that the first case can not happen since $P\cap X$ contains $L$ with multiplicity at least $3$. 

In the second case,  the multiplicity condition further implies that the quadric $q(x_1,x_2,x_3,x_4)$ does not have $x_1^2,x_1x_2,x_2^2$ terms. To see that there is no $x_1x_3$ term,  note that  the affine local equation ($y_i=x_i/x_0$) near $p_0$   can be written  as 
\[y_3^2+ by^2_1y_3+y_1y^2_3g(y_1,y_2,y_3)+y_1y_2^2\ell(y_1,y_2)+ay_1y_2y_3+c(y_2,y_3)=0\]
where $a,b\in\CC$, $g,c$ are polynomials  with $deg(c)\geq 3$.  Since $p_0$ deforms to $\widetilde{E}_8$,  we know that $a$ has to be $0$. 
\end{proof}


\begin{theorem}\label{thm:3}
A complete intersection $X=Q\cap Y$ is unstable if and only if $X$ satisfies one of the following conditions:
\begin{enumerate}[($i'$)]
  \item $X$ is singular along a line $L$ satisfying the condition:  there exist a plane $P$ such that $\PP(CT_p(X))=2P$ for any point $p\in L$;
    \item there exists a plane $P$ whose restriction to $X$ is  a line $L$ with multiplicity $6$  and $X$ has a  a corank $3$ singularity $p$ on $ L$. Moreover,  the projective tangent cone $\PP(CT_p(X))$ at $p$ is the union of the plane $P$ and a quadric surface and they meet at $L$ with multiplicity two.
\end{enumerate}
\end{theorem}
\begin{proof} We check the complete intersections of type $(U1)-(U2)$ case by case.

($i'$). To simplify the proof, we choose another monomial basis of $V$ as below: 
\begin{equation}\label{eq:3.7} 
 \cB':=\{x_0^{a_0}\ldots x_4^{a_4}|\sum\limits_{i=0}^4 a_i=3,~a_2\leq1\}.
\end{equation}   
Then the polynomial  of type $(U1)$ has the form 
\begin{equation}
x_0q_0(x_3,x_4)+x_1q_1(x_3,x_4)+x_2q_2(x_3,x_4)+c(x_3,x_4)=0.
\end{equation} 
At this time, $X$ is singular along the line $L:x_2=x_3=x_4=0$ and satisfies the condition described in ($i'$).

On the other hand,  the line $L$ on $Q$ can be written as $$L:x_2=x_3=x_4=0$$ for a suitable change of coordinates preserving $Q$. Then the equation of $Y$ has the form $$  \sum\limits_{i=0}^1 x_iq_i(x_2,x_3,x_4)+x_2q_2(x_3,x_4)+c(x_3,x_4)=0,$$ where $q_i$ does not contain $x_2^2$ term. 

Moreover, for any point $p=[z_0,z_1,0,0,0]\in L$,  the projective tangent cone $\PP(CT_p(X))$ is given by   $$z_0x_3+z_1x_4=z_0q_0(x_2,x_3,x_4)+z_1q_1(x_2,x_3,x_4)=0.$$  They have a common plane $P$ with multiplicity $2$ if and only if  $P$ is defined by $ x_3=x_4=0$ and $q_i(x_2,x_3,x_4)$ does not contain the $x_2x_3,x_2x_4$ terms.

($ii'$). When $Y$ has the equation 
$$x_4q_1(x_1,\ldots,x_4)+x_3q_2(x_2,x_3)+\lambda x_1x_3^2=0,$$
one observe that $X$ contains the line $L:x_2=x_3=x_4=0$ which is contained in the plane $P:=x_3=x_4=0$. It is easy to see that $P$ intersect with $X$ is the line $L$ with multiplicity $6$. Moreover, $X$ is singular at  $p_0=[1,0,0,0,0]$ and the projective cone at $p_0$ is given by
$$\{x_4=x_3q_2(x_2,x_3) +\lambda x_1x_3^2=0 \},$$
which is the union of the plane $X_1: x_3=x_4=0$ and the quadratic surface $X_2:x_4=q_2(x_2,x_3)+\lambda x_1x_3=0$ satisfying the desired conditions. The proof of the converse is quite similar as the previous cases and we omit the details here.

\end{proof}

\begin{corollary}
A complete intersection $X=Q\cap Y$ is semistable (resp. stable) if $X$ has at worst isolated singularities (resp.  simple singularities). 
\end{corollary}
\begin{proof}
By Theorem \ref{thm:3}, the singular locus of $X$ is at least one dimensional if it is unstable. Then $X$ has to be semistable if it has at worst isolated singularities.

Next, from Theorem \ref{thm:stable}, we know that if $X$ is not properly stable, then either $X$ is singular along a curve or it contains at least an isolated simple elliptic singularity. It follows  that  $X$ with simple singularities is stable. 
\end{proof}
Now it makes sense to talk about the moduli space $\cM_6$ of  complete intersections of a smooth quadric and a cubic  with simple singularities. Let $\cU_6$ be the open subset of $\PP(V)^s$ parameterizing such complete intersections in $\PP^4$.  Then we have $\cM_6=\cU_6/SO(5)(\CC)$.

\begin{theorem}\label{thm2.11}
There is an open immersion $\cP_6:\cM_6\rightarrow \cF_6$ via the period map and the image of $\cP_6$ in $\cF_6$ is the complement  of three NL-divisors $D^6_{1,1}, D^6_{2,1}$ and $D^6_{3,1}$. The Picard group $\Pic_\QQ(\cF_6)$ is spanned by $\{D^6_{d,1},~1\leq d\leq4\}$. 
\end{theorem}

\begin{proof}
For the first statement, one only need the fact that the complete intersections with simple singularities correspond to degree 6 quasi-polarized K3 surfaces containing a $(-2)$ curve. Therefore, we obtain an open immersion $\cP_6:\cM_6\rightarrow \cF_6$ from Torelli theorem. By Lemma \ref{lem:1}, we know that the boundary divisors of the image $\cP_6(\cF_6)$ is the union of  $D^6_{1,1}, D^6_{2,1}$ and $D^6_{3,1}$. 

Next, the moduli space $\cM_6$ is isomorphic to the quotient $\cU_6/SO(5)(\CC)$.  Observing that $\Pic(\cU_6)\cong \Pic(\PP(V))$ has rank one  since the boundary of $\cU_6$ in $\PP(V)$ has codimension at least two, we claim that the dimension of  $\Pic_\QQ(\cM_6)$ is at most one. Denote by $\Pic(\cU_6)_{SO(5)(\CC)}$ the set of $SO(5)(\CC)$-linearized line bundles on $\cU_6$. There is an injection 
 $$\Pic(\cU_6/SO(5)(\CC))\hookrightarrow \Pic(\cU_6)_{SO(5)(\CC)}$$ by \cite{KKV89} Proposition 4.2 for the reductive group $SO(5)(\CC)$. Our claim then follows from the fact  the forgetful map $\Pic(\cU_6)_{SO(5)(\CC)}\rightarrow \Pic(\cU_6) $ is an injection. Actually, one can easily see that $\Pic_\QQ(\cM_6)$ is spanned by the descent of the tautological line bundle $\cO_{\cU_6}(1)$ on $\cU_6$ to the quotient $\cU_6/SO(5)(\CC)$, and we denote it by $\cO_{\cM_6}(1)$.

Since the complement of $\cP_6(\cM_6)$  in $\cF_6$ is the union of three irreducible divisors and $\dim_\QQ(\Pic(\cF_6))\geq 4$, it follows that $\Pic_\QQ(\cF_{6})$ is spanned by  NL-divisors $\{D^{6}_{d,1}, 1\leq d\leq 4\}$ by the dimension consideration. 
\end{proof}

\begin{remark}
There is another natural GIT construction of moduli space of complete intersections in projective spaces, see \cite{Lo03, Be14}. There exists a projective bundle $\pi:\PP E\rightarrow \PP(H^0(\PP^5,\cO_{\PP^5}(2)))\cong \PP^{14}$ parameterizing all complete intersections of a quadric and a cubic in $\PP^5$. Then one can consider the GIT quotient $$\PP(E)/\!\!/_{H_t}SL_5(\CC)$$ for the line bundle $H_t=\pi^\ast \cO_{\PP^{14}}(1)+t\cO_{\PP E}(1)$. We want to point out that  $\PP(E)/\!\!/_{H_t}SL_5(\CC)$ is isomorphic to our GIT quotient  $\PP(V)/\!\!/SO(5)(\CC)$ when $t<1/6$. This will be discussed in the upcoming paper \cite{GLRST}. 
\end{remark}

%
%

\subsection{Minimal orbits} In this section, we give a  description of the semistable boundary components of the  GIT compactification. It consists of strictly semistable points with minimal orbits. From $\S 3.2$, it suffices to discuss the points  of type $(\alpha)-(\gamma)$. 
As in \cite{La09}, our approach is to use Luna's criterion:
\begin{lemma}{(Luna's criterion)}\cite{Lu75}
Let $G$ be a reductive group acting on an affine variety $V$. If $H$ is a reductive subgroup of $G$ and $x\in V$ is stabilized by $H$, then the orbit $G\cdot x$ is closed if and only if $C_G(H)\cdot x$ is closed.
\end{lemma}

To start with, we first observe that Type $(\alpha)$, $(\beta)$ and $(\gamma)$ have a common specialization, which we denote by Type $(\xi)$: 
 $$ \lambda_1 x_2^3+\lambda_2 x_1x_2x_3=0.$$

\begin{lemma}
If $X$ is of Type $(\xi)$, it is strictly semistable with closed orbits.
\end{lemma}
\begin{proof}
The stabilizer of Type $(\xi)$ contains a 1-PS: 
$$ H=\{diag(t^2,t,1,t^{-1},t^{-2})|~t\in \CC^\ast\},$$
 of distinct weights. So the  center 
$$C_G(H)=\{diag(a_0,a_1,1,a_1^{-1},a_0^{-1})\}\subset SO(Q)(\CC)$$
is a maximal torus. It acts on $V^H=\left<x_0x_3^2,x_1^2x_4,x_1x_2x_3,x_2^3\right>\subset V$.   It is straightforward to see any element of Type $(\xi)$ is semistable with closed orbit in $V^H$ under the action. Then the statement follows from Luna's criterion. 
\end{proof}

\begin{proposition}\label{alpha}
Let $X$ be a surface of Type $(\alpha)$. Then it has two corank $3$ singularities. Moreover, we have
\begin{enumerate}
\item  $X$ is unstable if it is union of a quadric surface and a quadric cone with multiplicity two.   
\item The orbit of $X$ is not closed if $X$ is singular along two lines. It degenerates to type $\xi$. 
\end{enumerate}
Otherwise, $X$ is semistable with closed orbit.  
\end{proposition}
\begin{proof} 
 The stabilizer of Type $(\alpha)$ contains a 1-PS: 
$$H_1=\{diag(t,1,1,1,t^{-1})|~t\in \CC^\ast\}. $$ 
The center $C_G(H_1)\cong SO(Q_1)(\CC)\times SO(Q_2)(\CC)$, where $Q_1=x_0x_4$ and $Q_2=x_1x_3+x_2^2$.  The group $SO(Q_1)(\CC)\cong SO(2;\CC)$ acts linearly on variable $x_0,x_4$,  while $SO(Q_2)(\CC)\cong SO(3)(\CC)$ acts linearly on the variables  $x_1,x_2$ and $x_3$. 

The action of $C_G(H_1)$ on $V^{H_1}=\left<x_1^{d_1}x_2^{d_2}x_3^{d_3}, \sum\limits_{k=1}^3 d_k =3\right>\subset V$ is equivalent to the action of $SO(Q_2)(\CC)$ on the set of cubic polynomials in three variables $x_1,x_2,x_3$ preserving the quadratic form $Q_2$.  By Luna's criterion, we can reduce our problem to a simpler GIT question $V^{H_1}/\!\!/SO(3)(\CC)$.   
Any $1$-PS $\lambda: \CC^\ast\rightarrow SO(Q_2)(\CC)$ of $SO(Q_2)(\CC)$ can be diagonalized in the form \begin{equation}\label{eq4.13}\lambda(t)= diag(t^a,1,t^{-a}).\end{equation} The weight of a monomial $x_1^{d_1}x_2^{d_2}x_3^{d_3}$ with respect to \eqref{eq4.13} is $a(d_1-d_3)$. Then our assertion follows easily from the Hilbert-Mumford criterion.
\end{proof}

The remaining cases can be shown in a similar way. Here we omit the proof.
\begin{proposition}\label{beta} Let  $X$ be a  surface of  type $(\beta)$. Then it is a union of  a quadric surface  and a complete intersection of two quadrics. Moreover, we have
\begin{enumerate}[(i)]
\item $X$ is unstable if $X$ consists of two quadric cones and a quadric surface intersecting at a line.
\item The orbit of  $X$ is not closed if its equation can be written as $ \lambda_1 x_2^3+\lambda_2 x_1x_2x_3+\lambda_3 x_1x_2x_4$ up to a coordinate transform preserving $Q$. It degenerates to type $(\xi)$.
\end{enumerate}
Otherwise, $X$ is semistable with closed orbit.  
\end{proposition}

\begin{proposition} \label{gamma}
A general member $X$ of type $(\gamma)$ has two simple elliptic singularities of type $\tilde{E}_8$. Moreover, we have
 \begin{enumerate}[(i)]
\item $X$ is unstable if $X$ consists of three quadric cones.
\item The orbit of $X$ is not closed if its equation has the form  $ \lambda_1 x_2^3+\lambda_2 x_1x_2x_3+\lambda_3 x_1^2x_4$ up to a coordinate change preserving $Q$.
\end{enumerate}
Otherwise, $X$ is semistable with closed orbit.  
\end{proposition}

\section{Stable singular complete intersection}
In this section, we will discuss the stable loci of singular complete intersections of a smooth quadric and a cubic hypersurface. As a result, we  prove Theorem \ref{thmboundary}. 
\subsection{Stable complete intersection with isolated singularity}
Let $X=Q\cap Y$ be a complete intersection of the smooth quadric $Q$ and a cubic threefold $Y$. A first observation is 
\begin{proposition}\label{ss-iso}
 If $X=Q\cap Y$ has only isolated singularity, then $X$ is stable if and only if the non-ADE singularities can only be  one of the following situations
 \begin{enumerate}
     \item [$i)$] $\widetilde{E}_7$ type, 
     \item [$ii)$ ]  $\widetilde{E}_8$ type and its projective tangent cone $\PP(CT_p)$ meet $X$ at a point, the general equation of $Y$ is of the form
 \end{enumerate}
Up to a coordinate change, the general equation of $Y$ for  $Y\cap Q$ with an $\widetilde{E}_7$ singularity is 
     \[(\delta):~~~x_0x_2^2+x_2q(x_1,x_2,x_3,x_4)+x_4q(x_1,x_3,x_4)=0\]
Similarly, the general equation of $Y$ for $Y\cap Q$ with an $\widetilde{E}_8$ singularity of type ii) is 
     \[(\epsilon):~~~x_0x_3^2+x_2x_3\ell(x_1,x_3,x_4) +c(x_1,x_3,x_4)=0\]
     where $\ell$ is linear and $c$ is a cubic polynomial in $x_1,x_3,x_4$. 
\end{proposition}

\begin{proof}
Suppose $X$ has only at worst isolated singularities of type $i)$ or $ii)$. This means that $X$ does not have a corank $3$ singularity or $\widetilde{E}_8$ singularity whose projective tangent cone meets $X$ along a line. By Theorem \ref{thm:stable}, we know that $X$ is stable. 
Conversely, suppose $X$ is stable and it has a non-ADE isolated singularity at $p=[1,0,0,0,0]$. If $p$ is not of type $i)$ or $ii)$, by Theorem \ref{thm:stable} (i),  $p$ has to be a $\widetilde{E}_8$ type singularity. Let us analysis the local equation of $p$. As in the proof in Theorem \ref{thm:stable}, up to a change of coordinates preserving the quadric $Q$, the defining equation of $Y$ has two possibilities: 
\begin{enumerate}
\item  $x_0x_2^2+x_1^2\ell(x_1,x_2,x_3,x_4)+x_1q(x_2,x_3,x_4)+c(x_2,x_3,x_4)=0$
\item $x_0x_3^2+x_1^2\ell(x_1,x_2,x_3,x_4)+x_1q(x_2,x_3,x_4)+c(x_2,x_3,x_4)=0$
\end{enumerate}

In the first case, the affine local equation near $p$ can be written as \[x_2^2+\sum\limits_{6\geq d\geq 3}f_d(x_1,x_2,x_3)=0\]
where $f_d$ is a homogenous polynomial of degree $d$. Note that there is no 
term  $x_1^3x_2, x_3^3x_2$ and $x_1x_3^4, x_1^4x_3$ in the fourth and fifth jet. 
One can easily see that $p$ can not  be a $\widetilde{E}_8$ type singularity. 

In the second case,  we know that the projective tangent cone $\PP(CT_p(X))$ meets $X$ along either a line $L:x_2=x_3=x_4=0$ or the point $p$ (with multiplicity). If the intersection is a line,  $X$ can not be stable by the proof in Theorem \ref{thm:stable}. The only possibility is that the intersection is a point. In this situation, the third jet contains the term $x_1^3$ and the weights on variables $x_1,x_2,x_3$ are $(\frac{1}{3}, \frac{1}{6}, \frac{1}{2})$.   The equation of $Y$ is of the form
\[x_0x_3^2+x_2x_3\ell(x_1,x_3,x_4)+c(x_1,x_3,x_4)=0,\]
where $\ell$ is linear  and $c$ is a cubic in $x_1,x_3,x_4$.

At the end, let us give the general equations for $X$ with an $\widetilde{E}_7$ singularity. Without loss of generality, we assume the singularity is at $p=[1,0,0,0,0]$. As it is corank $2$, the defining equation of $Y$ can be written as 
\[ x_0x_2^2+x_2q(x_1,x_2,x_3,x_4)+c(x_1,x_3,x_4)=0\]
or  \[ x_0x_3^2+x_3q'(x_1,x_2,x_3,x_4)+c'(x_1,x_2,x_4)=0\]
The weights on $(x_1,x_2,x_3)$ are either  $(\frac{1}{4}, \frac{1}{2}, \frac{1}{4})$ or $(\frac{1}{4},\frac{1}{4},\frac{1}{2})$. In either case, the direct computation shows that all monomials in $c(x_1,x_3,x_4)$ and $c'(x_1,x_2,x_4)$ must have $x_4$ term. The assertions follows. 
\end{proof}

From the proof, we can see that if $X$ has only isolated singularities of $\widetilde{E}_8$ type, then $X$ will be stable if the projective tangent cone $\PP(CT_p(X))$ of the singularity meets $X$ at a point and $X$ is strictly semistable if $\PP(CT_p(X))$ meets $X$ along a line. 

\subsection{Stable loci of complete intersection with non-isolated singularity}Let us consider the non-normal case. 
With the notations as above, we denote by  $\mathrm{Sing}(X)$  the singular loci of $X$. Then we have 
\begin{theorem}\label{ss-niso}
 Let $X$ be a complete intersection of a smooth quadric $Q$ and a cubic hypersurface $Y$ with non-isolated singularities. Then one of the following holds:
 \begin{enumerate}[i)]
     \item $\mathrm{Sing}(X)$ contains a line. The general equations of such $X$ are of the form
\[(\zeta): ~~~~x_0x_4+x_1x_3+x_2^2=x_3q_1(x_0,x_1,x_2)+x_4q_2(x_1,x_2)+c(x_0,x_1,x_2)=0\]
where $q_1,q_2$ are quadrics and $c$ is a cubic.
     \item $\mathrm{Sing}(X)$ contains a conic. The general equations of such $X$ are of the form
\[(\eta):~~~~x_0x_4+x_1x_3+x_2^2=x_0^2\ell_1+x_0x_4\ell_2+x_4^2\ell_3=0\]
where $\ell_i$ are linear polynomials in $x_0,\ldots,x_4$.
     \item $\mathrm{Sing}(X)$ contains a twisted cubic. The general equations of $X$ are of the form
\[(\theta): ~~~~\begin{aligned}
&Q:\begin{vmatrix}
x_{0} & x_1& x_2   \\
x_{1} & x_2&x_3  \\
a_1& a_2& a_3 
\end{vmatrix}+x_4\ell(x_0,x_1,x_2,x_4)=0\\ & Y:\begin{vmatrix}
x_{0} & x_1& x_2   \\
x_{1} & x_2&x_3  \\
b_1x_4& b_2x_4+a_1x_1+a_2x_0& b_3x_4+a_1x_2+a_2x_1+a_3x_0
\end{vmatrix}\\ &~~+x_4x_0\ell(x_1,x_2,x_3,x_4)+x_4^2\ell'=0\end{aligned}\] 
where $a_i,b_i\in\CC, \ell'$ is a linear polynomial in $x_0,\ldots,x_4$ and $\ell$ represents a linear polynomial in four variables. 
     \item $\mathrm{Sing}(X)$ contains an elliptic curve of degree four. The general equations of such $X$ are of the form
\[(\widetilde{\beta}):~~~ x_0x_4+x_1x_3+x_2^2=\ell q=0,\] 
where  $\ell$ is a linear polynomial and $q$ is a quadric polynomial. 
     \item $\mathrm{Sing}(X)$ contains a rational normal curve of degree four.  
     The general equations of such $X$ are of the form
\[(\phi): x_0x_4+x_1x_3-2x_2^2=\sum\limits_{i=2}^6 \ell_i \Delta_i=0\] 
where $\ell_i$ are linear polynomials in \eqref{linear} and $\Delta_i$ are quadric polynomials defined in \eqref{quartic}. 
\end{enumerate}
Moreover, the general members of each type is stable.  
\end{theorem}
\begin{proof}
Let $C\subseteq \mathrm{Sing}(X)$ be an irreducible curve.  If $X=X_1\cup X_2$ is reducible, then $\deg(X_i)=2$ or $4$ as $X_i$ is contained in a smooth quadric threefold $Q$. The only possibility is $X=Q\cap Y$ with $Y$ a union of a $\PP^3$ and a quadric threefold. This is exactly type $\widetilde{\beta}$. 

If $X$ is irreducible, take a general hyperplane $H$, then $H\cap X$ is an irreducible curve singular along $H\cap C$. Note that the arithmetic genus of $H$ is at most $4$,  $H\cap C$ has at most $4$  points. It follows that the degree of $C$ is at most $4$.  Hence $C$ can be a line, a conic, a plane cubic, a twisted cubic or a rational normal curve of degree $4$. If $C$ is a plane cubic, then $C$ is contained in the intersection $\PP^2\cap Q$, which is a conic. This is clearly impossible. Let us now describe their equations of  case by case. 

i) Take the quadric threefold $Q:x_0x_4+x_1x_3+x_2^2=0$ and  we can assume  $X$ is singular along the line $C:x_0=x_1=x_2=0$. The equation of $Y$ is of the form 
\begin{equation}\label{NLline} f=x_3^2\ell_1+x_3x_4\ell_2+x_4^2\ell_3+x_3q_1+x_4q_2+c=0
\end{equation}
where $\ell_i$ are linear, $q_i$ are quadric and $c$ is a cubic polynomial in $x_0,x_1$ and $x_2$. Then the Jacobian of $X$ on the line $C$  given by
\[ \begin{pmatrix}
x_4 & x_3& 0&0 &0 \\
\frac{\partial f}{\partial x_0} &  \frac{\partial f}{\partial x_1}  & \frac{\partial f}{\partial x_2}  & 0&0
\end{pmatrix} \]
has rank one. The only possibility is that all $\ell_i=0$. This gives the equation $(\xi).$
\\

ii) Take the quadric as above and we assume that $X$ is singular along a smooth conic $C:x_0=x_4=x_1x_3+x_2^2=0$.  The equation of a cubic hypersurface $Y$ containing $C$ is of the form
\begin{equation}\label{conic}
    x_0q_1+x_4q_2+(x_1x_3+x_2^2)\ell(x_1,x_2,x_3)=0
\end{equation}
for some quadric polynomials $q_1$ and $q_2$. Similarly as i), one can compute  that the equation of $Y$ is of type $(\eta)$.
\\

iii) Take $Q$ as above. If $C$ is an elliptic curve of degree $4$, the span of $C$ is a three dimension linear subspace, denoted by $H$. Note that $H\cap X$ can not be a curve as $C$ is contained in $H\cap X$ with multiplicity at least $2$. This means $H\cap X$ is a surface and thus $X$ is reducible. Such $X$ is of type $(\widetilde{\beta})$.
\\

iv) If $C$ is a twisted cubic, the defining equations of $C$ can be written as 
\[x_4=0, x_1x_3-x_2^2=0 , x_1x_2-x_0x_3=0, x_0x_2-x_1^2=0.\]
Then the equations of a complete intersection $X$ containing  $C$ can be written as 
\begin{equation}\label{twcubic}
    \begin{aligned}
   Q&:  \begin{vmatrix}
x_{0} & x_1& x_2   \\
x_{1} & x_2&x_3  \\
a_1& a_2& a_3 
\end{vmatrix} +x_4\ell=0
   \\ Y&: \begin{vmatrix}
x_{0} & x_1& x_2   \\
x_{1} & x_2&x_3  \\
\ell_1& \ell_2&\ell_3 
\end{vmatrix} +x_4q(x_0,x_1,x_2,x_3)+x_4^2\ell'=0.
    \end{aligned}
\end{equation}
for some linear polynomials $\ell_i, \ell,\ell'$ and a quadric polynomial $q$. If $X$ is singular along $C$, then via computing the Jacobian of equations \eqref{twcubic} , we get that the equation $(\theta)$.

vi) Without loss of generality, we can assume the curve $C$ is defined by the equations
\begin{equation}\label{quartic}
\begin{aligned}
\Delta_1=x_1x_3-x_2^2, \Delta_2=x_1x_2-x_0x_3, \Delta_3=x_0x_2-x_1^2;\\
\Delta_4=x_2x_4-x_3^2, \Delta_5=x_2x_3-x_1x_4, \Delta_6=x_1x_3-x_0x_4.
\end{aligned}
\end{equation}
and  $Q$ is defined by the equation $2\Delta_1-\Delta_6=0$. As $Y$ contains $C$, we may assume the equation of  $Y$ is given by
\begin{equation}\label{NLquartic}
   \sum\limits_{2}^6\ell_i\Delta_i=0
\end{equation}
for some linear polynomial $\ell_i$. As $X$ is singular along $C$,  the Jacobian matrix of $X$ along $C$ is 
\[\begin{pmatrix}
x_{4} & x_3& -4x_2& x_1&x_0   \\
\sum\limits_{i} \ell_i \frac{\partial \Delta_i}{\partial x_0} & \ldots & \ldots & \ldots & \sum\limits_i \ell_i \frac{\partial \Delta_i}{\partial x_4}
\end{pmatrix}\]
Then via a computation, one can get the linear functions $\ell_i$  are of the form
\begin{equation}\label{linear}
\begin{aligned}
& \ell_4=\sum\limits_{i=0}^2 a_i x_i, 
\ell_5=\sum\limits_{i=0}^2 b_ix_i+a_2x_3,\ell_2=\sum\limits_{i=1}^3 c_i x_i+(b_2-a_1)x_4 \\
&\ell_3=c_1x_2+(c_2-b_0)x_3+(c_3+a_0-b_1)x_4, \\
&\ell_6=\frac{1}{2}(c_1x_0+(c_2-2b_0)x_1+(c_3-2b_1+3a_0)x_2+(2a_1-b_2)x_3+a_2x_4). 
\end{aligned}
\end{equation}
There are nine parameters $a_i,b_i$ for $i=0,1,2$ and $c_j$ for $j=1,2,3$. We left the details to readers.

Finally, the assertion of stability follows directly from Theorem \ref{thm:stable}. 
\end{proof}

\noindent {\bf Proof of Theorem \ref{thmboundary}}. It basically follows from the  combination of  Proposition \ref{alpha}-\ref{gamma}, Proposition \ref{ss-iso} and Theorem \ref{ss-niso}. The strata $\alpha, \beta,\gamma$ are strictly semistable, which are  described in Proposition \ref{alpha}, Proposition \ref{beta} and Proposition \ref{gamma} respectively. The dimension of these components can be computed via Luna's slice theorem as below:

With the notations as in  Proposition \ref{alpha}-\ref{gamma}, we have 
\begin{equation}
\begin{aligned}
\dim (\alpha)&=\dim \PP(V^{H_1})\q SO(3)(\CC)=6
\\
\dim (\beta)&= \dim \PP(V^{H_2})\q \CC^\ast =2
\\
\dim (\gamma)&=\dim (\PP(V^{H_3})\q \CC^\ast=2
\end{aligned}
\end{equation}
where $V^{H_2}=\left<x_2^3, x_1x_2x_3,x_0x_2x_3,x_1x_2x_4\right>$ and $V^{H_3}=\left<x_2^3,x_1x_2x_3,x_0x_3^2,x_1^2x_4\right>$  parameterizing the equations of type $\beta$ and $\gamma$ respectively. 
For stable components, we can also compute the dimension as follows:
\begin{enumerate}
    \item For $\zeta$, note that the  NL-divisor $D^6_{1,0}$ parametrizing  $X$ containing a line has dimension $18$. The general member in $D^6_{1,0}$  is of the form \eqref{NLline}. Thus one can see that $\zeta$ has codimension $7$ in $D^6_{1,0}$  and it follows $\dim \zeta=11$.
    \item Similarly, the general equations of element in $D^6_{2,0}, D^6_{3,0}$ and $D^6_{4,0}$ are given in \eqref{conic}, \eqref{twcubic} and \eqref{NLquartic} respectively. Then one can directly see that  $\eta$ has codimension $7$ in $D^6_{2,0}$, while $\theta$ has codimension $15$ in $D^6_{3,0}$ and $\phi$ has codimension $16$ in $D^6_{4,0}$. 
    \item Let us consider $\widetilde{\beta}$ consisting of the union of a $\PP^2$ and a complete intersection of two quadrics meeting along a degree $4$ curve in $\PP^3$.  The element in $\widetilde{\beta}$ are parameterized by  the product of two projective spaces $\PP(V_1)\times \PP(V_2)$, where $V_1=H^0(Q,\cO_Q(1))$ and $V_2=H^0(Q,\cO_Q(2))$. Hence its dimension is 
      $$\dim \PP(V_1)+\dim \PP(V_2)-\dim SO(5)(\CC)=7.$$
      \item For $\delta$, it can be viewed as the quotient space $\PP(V)/G_1$, where $V$ is the vector space spanned by monomials in the equation $(\delta)$ and $G_1$ is the subgroup of $SO(5)(\CC)$ fixing the singular point $p_0$ and the hyperplane $x_2=0$. As  $\dim V=17$ and $\dim G_1=5$, we get $\dim \delta=11$. 
      \item Similar as above, $\epsilon$ is the quotient space $\PP(V')/G_2$ with $\dim V'=14$ and $\dim G_2=5$. It follows that $\dim \epsilon=8$. 
\end{enumerate}

\end{section}

\begin{section}{Complete intersection of three quadrics in $\PP^5$}
Let $W= H^0(\PP^5,\cO_{\PP^5} (2))$ be the space of global sections of $\cO_{\PP^5}(2)$.  Since every complete intersection $X$ is determined by a net of quadrics $Q_1,Q_2,Q_3$, the complete intersection of three quadrics are parametrized by the Grassmannian $Gr(3,W)$. The moduli space of complete intersections can be constructed as the GIT quotient $Gr(3,W)^{ss}/\!\!/SL_6(\CC)$. In this situation, the complete GIT strata is very complicated. For example, see \cite{DM12} for the GIT stability of a net of quadrics in $\PP^4$. However, we are satisfied with the following result:

\begin{theorem}\label{thm:ade-ss}
Let $X$ be a complete intersection of three quadrics in $\PP^5$. If $X$ has at worst simple singularities, then $X$ is GIT stable.
\end{theorem}

\subsection{Set up} We first make some notations. Given a net of quadrics $\{Q_1,Q_2,Q_3\}$, the Pl\uu cker coordinates of $\{Q_1,Q_2,Q_3\}$ in $\PP(\bigwedge^3 W)$ can be represented by 
$$\{x_{i_1}x_{j_1}\wedge x_{i_2}x_{j_2}\wedge x_{i_3}x_{j_3}\}$$ 
 for three distinct pairs $(i_k,j_k)$.
 
Let $\lambda:\CC^\ast \rightarrow SL_6(\CC)$ be a {\it normalized} one-parameter subgroup, i.e. $\lambda(t)=\hbox{diag}(t^{a_0},t^{a_1}\ldots, t^{a_5})$ satisfying
$ a_0\geq a_1\ldots \geq a_5$ and $\sum\limits_{i=0}^5 a_i=0$.  We denote by 
$$w_\lambda(x_ix_j):=a_i+a_j$$  
the weight of the monomial $x_ix_j$ with  respect to $\lambda$. The weight of a Pl\uu cker coordinate 
$x_{i_1}x_{j_1}\wedge x_{i_2}x_{j_2}\wedge x_{i_3}x_{j_3}$
 with respect to $\lambda$ is simply
$\sum\limits_{k=1}^3 w_\lambda(x_{i_k}x_{j_k})$. 

\subsection{Numerical Criterion for Nets}
By the Hilbert-Mumford numerical criterion, a net of quadrics $\{Q_1,Q_2,Q_3\}$ is not properly stable if and only if for a suitable choice of coordinates, there exists a normalized 1-PS $\lambda:t\rightarrow \hbox{diag}(t^{a_0},t^{a_1}\ldots, t^{a_5})$ such that the weight of all Pl\uu cker coordinates of $\{Q_1,Q_2,Q_3\}$ with respect to $\lambda$ is not positive. We say that $\{Q_1,Q_2,Q_3\}$ is not properly stable with respect to $\lambda$.

Given a normalized 1-PS $\lambda:\CC^\ast \to SL_6(\CC)$, we can define two complete orders on quadratic monomials:
\begin{enumerate}
\item $``>":x_0^2>x_0 x_1>\ldots >x_0x_5>x_1^2>x_1x_2>\ldots >x_4x_5>x_5^2$.
\item $``>_\lambda":x_ix_j>_\lambda x_kx_l$ if either $w_\lambda(x_ix_j)>w_\lambda(x_kx_l)$ or   $w_\lambda(x_ix_j)=w_\lambda(x_kx_l)$ for a given normalized 1-PS:$\lambda$ and $x_ix_j>x_kx_l$.
\end{enumerate}
Since the 1-PS $\lambda:\CC^\ast\rightarrow SL_6(\CC) $ is normalized,  $x_ix_j>_\lambda x_kx_l$ implies  $\max\{i,j\}>\min\{k,l\}$.

We denote by $m_i$  the leading term of $Q_i$ with respect to the order $``>_\lambda"$ and we say that a monomial $x_kx_l\notin Q_i$ if the quadratic polynomial $Q_i$ does not contain $x_kx_l$ term. 
 Moreover, we can always set 
 \begin{equation}\label{eq:4.0}
 m_1>_\lambda m_2>_\lambda m_3,
 \end{equation}
up to replacing $Q_1,Q_2,Q_3$ with a linear combination of the three polynomials. Then the term $m_1\wedge m_2\wedge m_3$ appears in the Pl\uu cker coordinates of $Q_1\wedge Q_2\wedge Q_3$ and has the largest weight with respect to $\lambda$.  Hence the net $\{Q_1,Q_2,Q_3\}$ is not properly stable with respect to $\lambda$ if and only if $w_\lambda(m_1\wedge m_2\wedge m_3)\leq 0$.

\begin{lemma}\label{lem:6}  With the notation above, let $X$ be the complete intersection $Q_1\cap Q_2\cap Q_3$. Then
$X$ has a singularity with multiplicity greater than two if  one of the following conditions does not hold:

\begin{enumerate}[(1)]
\item $m_1\geq_\lambda x_0x_4$,

\item $m_2\geq_\lambda x_1x_5$ if $m_1=x_0^2$, and $m_2\geq_\lambda x_0x_5$ otherwise,

\item $m_3 \geq_\lambda x_3^2$ if $m_1<_\lambda x_0x_3$.
\end{enumerate}
Moreover, $X$ is singular along a curve if one of the following conditions does not hold:

\begin{enumerate}[(1')]

\item $m_1\geq_\lambda x_1^2$  if $m_3<_\lambda x_1x_5$;  or  $ m_1\geq_\lambda \max\{x_1x_3,x_2^2\}$ if $m_2<_\lambda x_1x_4$;

\item $m_2\geq_\lambda x_2^2$ if $m_3<_\lambda x_2x_5$; $m_2\geq_\lambda \max \{x_1x_4,x_3^2\}$ if $m_1<_\lambda x_1^2$; and  $m_2\geq_\lambda \max \{x_2x_4,x_3^2\}$ otherwise;

\item $m_3\geq_\lambda \max\{x_3x_5,x_4^2\}$.
\end{enumerate}
\end{lemma}

\begin{proof}
Let $p_0$ be the point $[1,0,0,0,0,0]$ in $\PP^5$. For  (1) and (2),  if either $m_1<_\lambda x_0x_4$ or $m_2<_\lambda x_0x_5$ and $m_1<_\lambda x_0^2$,  the surface $X$ contains the point $p_0$ and  two quadrics $Q_2,Q_3$ are both singular at $p_0$. It follows that multiplicity of $p_0$ is greater than $2$.  

If $m_1=x_0^2$ and $m_2<_\lambda x_1x_5$, then $X$ is singular along the two points
 $$\{Q_1=x_2=x_3=x_4=x_5=0\}$$ with multiplicity greater than $2$. Similarly, one can easily check our assertion for (3). 

For (1'), (2') and (3'),  we will only list the singular locus of $X$ and leave the proof to readers:
\begin{itemize}
\item $X$ is singular along the line $L:x_2=x_3=x_4=x_5=0$  if  condition $(1')$ is invalid.

\item $X$ is either reducible or singular along $L$ or $C_1: x_3=x_4=x_5=Q_1=0$ if  condition $(2')$ is invalid.

\item $X$ is either reducible or singular along the curve $C_2:x_4=x_5=Q_1=Q_2=0$ if  condition $(3')$ is invalid.
\end{itemize}
\end{proof}

As before, we need to know the maximal set $M_{\leq 0}(\lambda)$ of  triples  of distinct quadratic monomials $\{q_1,q_2,q_3\}$, whose sum of their weights with respect to $\lambda$ is non-positive.  Instead of looking at all maximal subsets, we are interested in the maximal subset $\overline{M}_{\leq 0}(\lambda)$ which contains a triple $\{m_1,m_2,m_3\}$ satisfying the conditions $(1)-(3)$ and $(1')-(3')$ in Lemma \ref{lem:6}. It is not difficult to compute that there are four such maximal subset. See Table \ref{tab:3} below. 

\begin{table}[ht]\label{tab:3}
  \centering
  \caption{Maximal set $\overline{M}_{\leq 0}(\lambda)$}
  \begin{tabular}{|c|c|c|c|c|c|c|c|}
    \hline
  \multirow{2}{*}{Cases}  &\multirow{2}{*}{$\lambda=(a_0,\ldots,a_5)$} &\multicolumn{3}{c|}{Maximal triples $\{q_1,q_2,q_3\}$}\\
    \cline{3-5}
   ~&~&$q_1$& $q_2$& $q_3$\\
    \hline
    $(N1')$&$(2,1,0,0,-1,-2)$& $x_0x_2,x_1^2$& $x_0x_5,x_1x_4,x_2^2$ & $x_2x_5,x_4^2$\\ \hline
    $(N2')$& $(3,1,1,-1,-1,-3)$ &$x_0x_3,x_1^2$& $x_0x_5,x_1x_3$&  $x_1x_5,x_3^2$\\ \hline
    $(N3')$& $(4,1,1,-2,-2,-2)$& $x_0x_3,x_1^2$ & $x_0x_3, x_1^2$ & $x_3^2$ \\ \hline
     $(N4')$& $(5,3,1,-1,-3,-5)$& $x_0x_4,x_1x_3,x_2^2$ & $x_0x_5,x_1x_4,x_2x_3$ & $x_1x_5,x_2x_4,x_3^2$ \\ \hline
         \end{tabular}
\end{table}

The lemma below gives a geometric description of $X$ of type $(N1')-(N4') $.
\begin{lemma}\label{iso-X-ss}
Let $X$ be a general element of type $(N1')-(N4')$. Then $X$ has an isolated simple elliptic singularity. 
\end{lemma}
\begin{proof}
Obviously, $X$ is singular at $p_0=[1,0,0,0,0,0]$. Moreover,  $p_0$ is an isolated hypersurface  singularity when $X$ is general. To show it is simple elliptic, let us compute the analytic type of  $p_0$  case by case.

If $X$ is a general element of type $(N1')$, then the equations of $Q_i$ can be written as 
\begin{equation*}
\begin{aligned}
Q_1: &~x_0x_2+q(x_1,\ldots, x_5)=0 \\
Q_2:  &~x_0x_5+ x_1x_4+q'(x_2,x_3,x_4,x_5)=0\\
Q_3: &~x_4^2+x_5 l(x_2,x_3,x_4,x_5)=0
\end{aligned}
\end{equation*}
up to a linear change of the coordinates. Let us take the local coordinates near $p_0$:
\begin{equation}\label{eq:4.2}y_i=x_i/x_0.
\end{equation}  From the first two quadratic equations, one can get 
\begin{equation*}
\begin{aligned}
y_2&=f_1(y_1,y_3,y_4), \\
 y_5&=y_1y_4+ by^2_3+b'y_3f_1(y_1,y_2,y_4)+f_2(y_1,y_3,y_4),
\end{aligned}
\end{equation*}
for some formal power series $f_1\in \CC[[y_1,y_3,y_4]]_{\geq 2} , f_2\in \CC[[y_1,y_3,y_4]]_{\geq 4}$ and some constant $b,b'\in\CC$. Therefore, the local equation of $p_0$ is 
\begin{equation} 
y_4^2+ \alpha_1 y_3^3+\alpha_2 y_3^2y_1^2+\alpha_3 y_3y_1^4+\alpha_4 y_1^6 +(\geq \text{higher order terms})=0,
\end{equation} 
for some complex number $\alpha_i$. According to $\S$3.1, the singularity $p_0$ is simple elliptic of type $\tilde{E}_8$.

If $X$ is a general element of type $(N2')$,  we write the equations as 
\begin{equation*}
\begin{aligned}
Q_1: &~x_0x_3+q(x_1,\ldots,x_5)=0 \\
Q_2:  &~x_0x_5+ x_1x_3+x_2x_4=0\\
Q_3: &~q'(x_3,x_4,x_5)+x_5l(x_1,x_2)=0
\end{aligned}
\end{equation*}
Still, we  take the affine coordinate (\ref{eq:4.2}) near $p_0$ and then  we have 
\begin{equation*}
y_3=f(y_1,y_2,y_4), ~~ y_5=-y_1f(y_1,y_2,y_4)-y_2y_4,
\end{equation*}
for some $f\in \CC[[y_1,y_2,y_4]]_{\geq 2}$.
Thus the local equation around $p_0$ is
\begin{equation} 
\alpha y_4^2+ g(y_1,y_4) +y_4 g'(y_1,y_2,y_4)=0.
\end{equation} 
where $g\in \CC[[y_1,y_2]]_{\geq 4},$ $g'\in \CC[[y_1,y_2,y_4]]_{\geq 2}$ and $\alpha\in\CC$ is a constant. Hence $p_0$ is simple elliptic of type $\tilde{E}_7$ by $\S$3.1.

One can similarly prove that $X$ has a simple elliptic singularity $p_0$ of type $\tilde{E_7}$ when it is general of type $(N3')$,  and of type $\tilde{E}_8$ when it is general of type $(N4')$.
\end{proof}

\subsection{Image of the period map} Let $\cU_8\subset Gr(3,W)$ be the open subset consisting of all complete intersections with at worst simple singularities. By Lemma \ref{iso-X-ss},  we know that $\cU_8$ is contained in the stable locus of $ Gr(3,W)$. This proves Theorem \ref{thm:ade-ss}. Moreover,  similarly as Theorem \ref{thm2.11}, we can get the following result

\begin{theorem}\label{thm3.4}
Let $\cM_8=\cU_8/\!\!/ SL_6(\CC)$ be the moduli space of the complete intersection of three quadrics in $\PP^5$ with simplest simple singularities. Then
\begin{enumerate}
    \item[(i)] the boundary of $\cM_8$ in $\overline{\cM}_8$ has codimension $\geq 2$. 
    \item[(ii)] there is an open immersion $\cP_8:\cM_8\rightarrow \cF_8$ as the extended period map and  the complement of $\cP_8(\cM_8)$ in $\cF_8$ is the union of three NL-divisors $D^8_{1,1}, D^8_{2,1}$ and $D^8_{3,1}$. The Picard group $\Pic_\QQ(\cF_8)$ is spanned by $\{D^8_{d,1},~1\leq d\leq4\}$. 
\end{enumerate}
\end{theorem}
\begin{proof}
For (i),  let $\Delta\subseteq Gr(3,W)$ be the discriminant divisor which parameterizes singular  complete intersections.  Then $\Delta$ is $SL_6(\CC)$-invariant and irreducible (cf.~\cite{GLT15}). Consider the GIT quotient $\Delta/\!\!/ SL_6(\CC)$.  By Theorem \ref{thm:ade-ss}, the general members in $\Delta$ is stable,  so the boundary $\overline{\cM}_8\backslash \cM_8$ lies in the boundary of $\Delta/\!\!/ SL_6(\CC)$ as a proper closed subset. It follows that $\overline{\cM}_8\backslash \cM_8$ has codimension two in $\overline{\cM}_8$.

For (ii), this follows from the same argument as in Theorem \ref{thm2.11}.
\end{proof}
\end{section}

\section{Arithmetic compactification of locally Hermitian symmetric varieties}  Baily and Borel compactify the arithmetic quotient $\Gamma_{2\ell}\backslash\DD$ to a normal projective variety $\overline{\Gamma_{2\ell}\backslash\DD}^{bb}$ by adding finitely many modular curves and singletons, which correspond to the classes of $\QQ$-isotropic subspaces of $\Lambda_{2\ell}^\CC$ of dimension $2$ and $1$. In \cite{Lo03}, Looijgenga gives an arithmetic compactification of the completment of hyperplane arrangements in $\Gamma_{2\ell}\backslash\DD$ in the spirit of Satake-Baily-Borel theory.  In our situation, the hyperplane arrangement we are interested in will be the union of three NL-divisors $D^{2\ell}_{d,1}$ for $d=1,2,3$ 

\subsection{A review of Looijenga's work}
Let $\fE$ be a collection of elements in $\Lambda$. The orthogonal complement of  $ \beta\in \fE$ and $h_{2\ell}$ in $\Lambda$ defines a hyperplane  $H_{\beta}\subseteq \PP(\Lambda_{2\ell}^\CC)$. Set $\DD_{H_\beta}=\DD\cap H_{\beta}$ to be the hyperplane arrangement and 
define $$\DD^\circ_{\fE}=\DD-\bigcup\limits_{\beta\in \fE} \DD_{H_\beta}$$
to  be the complement of all subdomains obtained from $\fE$. The quotient $\Gamma_{2\ell}\backslash \DD^\circ_{\fE}$ is the complement of Heenger divisors.

Looijenga constructed the compactification $\widehat{\Gamma_{2\ell}\backslash\DD^\circ_{\fE}}$ from the strata of decomposition of rational cones. It can be also viewed as the natural blowdown of certain minimal normal blowup of the Baily-Borel compactification $\overline{\Gamma_{2\ell}\backslash\DD}^{bb}$. The structure of the birational map is explicitly provided that how the hyperplanes $H_\beta$ intersect inside the period domain $\DD$. To make it precise, we fix our temporary notation as follows: 
\begin{itemize}
\item ${\rm PO}(\fE)$: the collection of subspaces $M\subseteq \Lambda_{2\ell}^\CC$ which are intersection of the hyperplane arrangements from $\fE$. Denote by  $$\pi_M: \PP(\Lambda_{2\ell}^\CC)-\PP(M)\longrightarrow \PP(\Lambda^\CC_{2\ell}/M)$$  the natural projection.  The projection also defines a natural subdomain $\pi_M\DD^\circ_\fE\subseteq \DD^\circ_\fE$ (cf.~\cite{Lo03} $\S$7). 
	
\item ${\rm I}(\fE)$: the collection of the common intersection of $I^\perp$ and hyperplane arrangements from $\fE$ containing $I$, where $I$ is a $\QQ$-isotropic subspace of $\Lambda_{2\ell}^\CC$.
\end{itemize}

The compactification $\widehat{\Gamma_{2\ell}\backslash\DD^\circ_{\fE}}$ can be interpreted as below: 
we define 
\begin{equation}
\widehat{\DD}_\fE=\DD^\circ_\fE\cup \coprod\limits_{M\in {\rm PO}(\fE)} \pi_M\DD^\circ_{\fE} \cup \coprod\limits_{V\in {\rm I}(\fE)} \pi_{V}\DD^\circ_{\fE}
\end{equation}
then the compactification  $\widehat{\Gamma_{2\ell}\backslash\DD^\circ_{\fE}}$ is isomorphic to the quotient $\Gamma_{2\ell}\backslash\widehat{\DD}$, and boundary thus decomposes into finitely many strata. A consequent of this description is that if the $r$-th self intersection of hyperplane arrangements $(\Gamma_{2\ell}\backslash \bigcup_\beta \DD_{H_\beta})^{(r)}\neq \emptyset$, then  \begin{equation}\label{dim}
\dim (\widehat{\Gamma_{2\ell}\backslash\DD^\circ_{\fE}} - \Gamma_{2\ell}\backslash\DD^\circ_{\fE})\geq r+1.
\end{equation}

When the codimesion of $\widehat{\Gamma_{2\ell}\backslash\DD^\circ_{\fE}}-\Gamma_{2\ell}\backslash\DD^\circ_{\fE}$ is greater than one, there is an explicit description of  
$ \widehat{\Gamma_{2\ell}\backslash\DD^\circ_{\fE}} $ in terms of the algebra of automorphic forms:  let $\LL$ be  natural automorphic line bundle on $\DD$ and $\LL^\circ$ the restriction of $\LL$ to $\DD^\circ_{\fE}$, then 
\begin{equation}\label{Locom}
\widehat{\Gamma_{2\ell}\backslash\DD^\circ_\fE}\cong \Proj\bigoplus\limits_{k\in\ZZ}H^0(\DD^\circ_\fE, (\LL^\circ)^{\otimes k})^{\Gamma_{2\ell}},
\end{equation}
by Corollary 7.5 in \cite{Lo03}.

\subsection{Application to moduli problem via GIT}Let us discuss the possible geometric interpretation of $\widehat{\Gamma_{2\ell}\backslash\DD^\circ_{2\ell}}$.  For many known geometric examples, such as Enrique surface, K3 surfaces of degree 2 and cubic fourfolds, the natural GIT compactification is precisely  Looijenga's compactification (cf.~\cite{Lo03,La10}). 
So it is interesting to investigate the relations between the compactifiations from GIT and arithmetic for K3 surfaces. For K3 surface with Mukai models, Laza has first found that the two compactifications do not necessarily coincide.  This fails for quartic surfaces in $\PP^3$.  
We can show that this actually happens quite often.

In our case,  let $\fE_{2\ell}$ be the collection of elements $\beta\in \Lambda $ satisfying $\beta^2=0$ and $\beta\cdot h_{2\ell}=1,2$ or $3$. Then $\Gamma_{2\ell}\backslash\DD_{\fE_{2\ell}}^\circ$ is the completement of three NL-diviosrs $D^{2\ell}_{d,1}$ for  $d=1,2,3$. The following lemma gives a rough description of the dimension of the boundary strata of $\widehat{\Gamma_{2\ell}\backslash \DD^\circ_{\fE_{2\ell}}}$:
\begin{lemma}\label{boundaryofk3}
When $2\ell=6$ and $8$,	the boundary of $\widehat{\Gamma_{2\ell}\backslash\DD^\circ_{\fE_{2\ell}}}-\Gamma_{2\ell}\backslash\DD^\circ_{\fE_{2\ell}}$ has codimension $1$.
\end{lemma}

\begin{proof} To understand the dimension of boundary, it suffices to consider the intersection of hyperplane arrangements from $\fE_{2\ell}$. Let $M$ be an even lattice of signature $(1,17)$ spanned by $h_{2\ell}$ and elements $e_1,e_2,\ldots, e_{17}$ satisfying that  $e_i^2=0,e_ie_j=1$ and $h_{2\ell}e_i=3$ for $i\neq j$.  It is easy to check this lattice has signature $(1,17)$ and thus can be embedded into $\Lambda$. Then the lattice $M$ can represent the intersection of  $17$ hyperplane arrangements from $\fE_{2\ell}$.  This proves the assertion by \eqref{dim}. 
\end{proof}

\begin{corollary}
For $2\ell=6$ and $8$, the GIT quoitent $\overline{\cM}_{2\ell}$ is not isomorphic to Looijenga's compactification. 
\end{corollary}
\begin{proof}Since  $\cM_{2\ell}$ is isomorphic to $\Gamma_{2\ell}\backslash \DD_{\fE_{2\ell}}^\circ$,  
this is obtained by comparing the dimension of the  boundary of $\overline{\cM}_{2\ell}-\cM_{2\ell}$ and $\widehat{\Gamma_{2\ell}\backslash\DD^\circ_{\fE_{2\ell}}}-\Gamma_{2\ell}\backslash\DD^\circ_{\fE_{2\ell}} $.
\end{proof}
In general, we believe that natural GIT compactifications of K3 surfaces with Mukai models constructed in \cite{GLT15} will not be the same as Looijenga's compactification. This can be achieved by a similar method. 

Another interesting problem is to study the birational maps between  $\widehat{\Gamma_{2\ell}\backslash\DD^\circ_{\fE_{2\ell}}}$ and $\cM_{2\ell}$. 
In a sequel to this paper, the authors together with Greer and Laza will study the birational geometry of $\cF_6$ via the variation of GIT and Looijenga's arithmetic approach. 
\bibliographystyle {plain}
\bibliography{MLK}

\end{document}